\documentclass[journal]{IEEEtran}
\usepackage[utf8]{inputenc}
\usepackage{amsmath, amsfonts, amsthm, amssymb}
\usepackage{graphicx}
\usepackage[margin=1.0in]{geometry}
\usepackage{biblatex}
\bibliography{main}

\newtheorem{theorem}{Theorem}
\newtheorem{lemma}{Lemma}

\title{Measurement-Feedback Control with \\ Optimal Data-Dependent Regret}
\author{Gautam Goel and Babak Hassibi
\thanks{Gautam Goel is with the Department of Computing and Mathematical Sciences at Caltech (e-mail: ggoel@caltech.edu). }
\thanks{Babak Hassibi is with the Department of Electrical Engineering at Caltech (e-mail: bhassibi@caltech.edu).}}
\date{}
\begin{document}
\maketitle

\begin{abstract}
    Inspired by online learning,  \textit{data-dependent regret} has recently been proposed as a criterion for controller design. In the regret-optimal control paradigm, causal controllers are designed to minimize regret against a hypothetical \textit{optimal noncausal controller}, which selects the globally cost-minimizing sequence of control actions given noncausal access to the disturbance sequence.
    We extend regret-optimal control to the more challenging \textit{measurement-feedback} setting, where the online controller must compete against the optimal noncausal controller without directly observing the state or the driving disturbance. We show that no measurement-feedback controller can have bounded competitive ratio or regret which is bounded by the pathlength of the measurement disturbance. We do derive, however, a controller whose regret has optimal dependence on the joint energy of the driving and measurement disturbances, and another controller whose regret has optimal dependence on the pathlength of the driving disturbance and the energy of the measurement disturbance. The key technique we introduce is a reduction from regret-optimal measurement-feedback control to $H_{\infty}$-optimal measurement-feedback control in a synthetic system. We present numerical simulations which illustrate the efficacy of our proposed control algorithms.
\end{abstract}

\section{Introduction}


Inspired by online learning,  \textit{data-dependent regret} has recently been proposed as a criterion for controller design \cite{goel2022regret}. In the regret-optimal control paradigm, causal controllers are designed to minimize regret against a hypothetical \textit{optimal noncausal controller}, which selects the globally cost-minimizing sequence of control actions given noncausal access to the disturbance sequence. Controllers with low regret retain a performance guarantee relative to this strong benchmark irrespective of how the disturbance is generated; it is this universality which makes regret-optimal controllers an attractive alternative to traditional $H_2$ and $H_{\infty}$ controllers, which instead posit a specific disturbance-generating mechanism. The regret of the causal controller is bounded by some measure of the complexity of the disturbance sequence; several different complexity measures have been proposed, including the energy of the disturbance sequence, which measures the size of the disturbance, and the pathlength of the disturbance, which measures its variation over time. The alternative metric of \textit{competitive ratio}, which is the worst-case ratio between the cost incurred by the causal controller and the cost incurred by the optimal noncausal controller, has also been proposed as a performance objective \cite{goel2019beyond}. This metric can also be viewed as a special case of data-dependent regret, where the complexity measure is simply the offline optimal cost. 

In a series of papers, Goel and Hassibi showed that for each of these complexity measures, the corresponding full-information controller can be computed via a reduction to full-information $H_{\infty}$ control \cite{goel2021competitive,  goel2022online, goel2021regret,}. In this paper we consider the more challenging problem of regret-optimal measurement-feedback control. In this setting, the causal controller is unable to directly observe the state or disturbance, and instead only has access to noisy linear observations of the state. As a result, the controller must not only choose which control actions to select, but must maintain and update an estimate of the state as new measurements arrive. In the $H_2$ setting it is well-known that the measurement-feedback problem can be decomposed into a stochastic full-information optimal control problem and a stochastic filtering problem; the fact that stochastic control algorithm and the filtering algorithm can be designed separately is usually called the Separation Principle. The $H_{\infty}$-optimal measurement-feedback controller can also be reduced to a filtering and full-information control problem, but the filtering and control algorithms are strongly coupled \cite{hassibi1999indefinite}. Intuitively, this is because an adversary which can optimize over the driving disturbance and the measurement noise simultaneously is strictly more powerful than one which must choose the disturbances separately. 

We present the first characterization of measurement-feedback controllers with optimal regret against the optimal noncausal controller.  We show that general linear systems do not admit measurement-feedback controllers with bounded competitive ratio or regret which is bounded by the pathlength of the measurement disturbance. We do, however, derive two new regret-optimal measurement-feedback controllers, the first of which attains regret with optimal dependence on the 
joint energy of the driving disturbance and the measurement disturbance, and the second of which attains regret with optimal dependence on the 
pathlength of the driving disturbance and the energy of the measurement disturbance. In both cases we reduce the problem of finding a controller with optimal regret in the original system to the problem of finding an $H_{\infty}$-optimal measurement-feedback controller in a synthetic system. We present numerical simulations which demonstrate that our regret-optimal controllers are competitive with standard $H_2$ and $H_{\infty}$ controllers.

\section{Related Work}
The problem of obtaining a controller whose regret has optimal dependence on the energy of the disturbance was studied in \cite{goel2021regret, sabag2021regret}. 
Recently, \cite{didier2022system} and \cite{martin2022safe} described connections between regret-optimal control and the System-Level-Synthesis (SLS) framework \cite{anderson2019system}, and used these connections to obtain regret-minimizing controllers that satisfy safety constraints. We also note the paper \cite{goel2022online}, which obtained a controller whose  regret has optimal dependence on the pathlength of the disturbance.
The problem of designing a controller with bounded competitive ratio was first studied in \cite{goel2019online}, and those results were extended in \cite{goel2019beyond, shi2020online}. A controller with optimal competitive ratio was obtained in \cite{goel2021competitive}. We also note a parallel line of work \cite{agarwal2019online, simchowitz2020making} which studies control through the weaker metric of \textit{policy regret}; this metric compares the performance of the online controller to the best controller selected in hindsight from a parametric class of controllers, such as the class of state-feedback controllers. We note that this benchmark is strictly weaker than the the optimal noncausal controller we consider in this paper, as established in \cite{goel2020power}.

All of these aforementioned works focus on the full-information control setting; \cite{goel2021measurement} was the first to study dynamic regret minimization in the measurement-feedback setting. In this paper, we focus on obtaining a measurement-feedback controller which minimizes regret against a clairvoyant, globally-optimal noncausal controller which observes the actual driving disturbance without any measurement noise; the paper \cite{goel2021measurement} instead considers the problem of minimizing regret against a weaker $H_2$-optimal noncausal benchmark.

\section{Preliminaries}
We consider a discrete-time LTI system that evolves according to the dynamics $$x_{t+1} = Ax_t + B_u u_t + B_w w_t,$$ where $x_t \in \mathbb{R}^n$ is the state, $u_t \in \mathbb{R}^m$ is the control input, $w_t \in \mathbb{R}^p$ is an exogenous disturbance, and $A, B_u, B_w$ are matrices of compatible dimensions. In each timestep, the controller receives a noisy linear measurement $y_t$ of the current state, where $y_t =  C x_t + v_t$ and  $C \in \mathbb{R}^{r \times n}$ and $v_t \in \mathbb{R}^r$ is a \textit{measurement disturbance}.
The controller also incurs a quadratic cost $$x_t^* Q x_t + u_t^* R u_t $$ in each timestep; it is convenient to assume that the dynamics are scaled so that $R = I$.  Our goal is to design a controller which, given the measurements $y$, selects a control action $u$ so as to minimize the dynamic regret against a clairvoyant offline optimal controller which can observe the full sequence of disturbances $w$ ahead of time.  Define $L = Q^{1/2}$, $s_t = L x_t$. The dynamics and observation model can be neatly expressed $$ s = Fu + Gw, \hspace{3mm} y = Hu + Jw + v,$$ where $F$ and $G$ are the strictly causal transfer operators mapping $u$ and $w$ to $s$ and $H$ and $J$ are the strictly causal transfer operators mapping $u$ and $w$ to $y$.
We restrict our attention to control policies which are a causal linear function of the observations $y$, i.e., policies which set $u = Ky$ for some causal operator $K$. Solving for $y$, we see that $$y = (I - HK)^{-1}(Jw + v), $$ implying that $$u =  K(I - HK)^{-1}(Jw + v).$$ We introduce the Youla parameterization $Q = K(I - HK)^{-1}$; we can easily recover $K$ from $Q$ by setting $K =  (I + QH)^{-1}Q$. To each $K$, we associate the transfer operator  $$T_K: \begin{bmatrix} w \\ v \end{bmatrix} \rightarrow \begin{bmatrix} s \\ u \end{bmatrix} $$ given by 
\begin{equation} \label{transfer-operator}
T_K = \begin{bmatrix} G & 0 \\ 0 & 0 \end{bmatrix}  + \begin{bmatrix}F \\ I \end{bmatrix} Q \begin{bmatrix} J & I \end{bmatrix}.
\end{equation}
Recall that the optimal noncausal controller selects the control $u = K_0 w$, where $$K_0 = -(I + F^*F)^{-1}F^*G.$$ The transfer operator associated to the optimal noncausal controller is therefore  $$T_{K_0} = \begin{bmatrix} -F(I + F^*F)^{-1}F^*G + G & 0 \\ -(I + F^*F)^{-1}F^*G & 0 \end{bmatrix}. $$
The zeros in the second column represent the fact that the optimal noncausal controller observes the actual disturbance $w$ and the control signal it selects is not at all affected by the measurement noise $v$; it is this disparity between the information available to the online and offline controllers which makes regret-optimal measurement-feedback control considerably more challenging than regret-optimal full-information control. 

The \textit{regret} of a causal controller $K$ against the optimal noncausal controller $K_0$  on the instance $(w, v)$ is simply the difference in their costs: $$\begin{bmatrix} w \\ v \end{bmatrix}^* (T_K^*T_K - T_{K_0}^*T_{K_0}) \begin{bmatrix} w \\ v \end{bmatrix}.$$ Similarly, the competitive ratio of $K$ is the worst-case ratio of their costs: $$\sup_{w, v} \, \frac{\begin{bmatrix} w \\ v \end{bmatrix}^* T_K^*T_K  \begin{bmatrix} w \\ v \end{bmatrix}}{\begin{bmatrix} w \\ v \end{bmatrix}^* T_{K_0}^*T_{K_0} \begin{bmatrix} w \\ v \end{bmatrix}}. $$
Recall that the energy of a signal $w$ is  $\|w\|^2$, while the pathlength of $w$ is $\sum_{t=-\infty}^{\infty} \|w_t - w_{t-1}\|^2$.
The goal of this paper is to completely characterize the conditions under which it is possible to obtain a controller with bounded competitive ratio, or regret bounded by the energy or pathlength of $w$ and $v$. In those cases where it is possible to obtain a controller with bounded regret, we derive controllers with the tightest possible regret bounds.

\section{Non-existence results} \label{non-existence-sec}

We first establish that in unstable systems there is no controller with bounded competitive ratio; in stable systems the only such controller is the controller which always sets $u = 0$. The key observation is that the competitive ratio bound implies that the online controller must incur zero cost on every disturbance sequences on which the optimal noncausal controller incurs zero cost. In particular, in order to have a bounded competitive ratio, a controller must always set the control $u$ to be zero when $w = 0$. In the measurement-feedback setting the controller is unable to directly observe $w$, and hence must set $u = 0$ at all times. The only setting under which this controller can be competitive is when the system is stable.

\begin{theorem}
Fix $\gamma > 0$ and suppose $A$ is unstable. There does not exist a causal measurement-feedback controller $\pi$ such that 
\begin{equation} \label{competitive-measurement-feedback-cond}
J(\pi, w, v) < \gamma^2 J(\pi_0, w)
\end{equation}
for all driving disturbances $w$ and and all measurement disturbances $v$. If $A$ is stable, then the only competitive controller is the controller which always sets $u = 0$. This controller has competitive ratio $1 + \|F\|^2$.
\end{theorem}
\begin{proof}
Suppose there exists  a measurement $y_0$ such that $\pi(y_0) \neq 0$. Set $w = 0$, $v = y_0$. The online controller incurs positive cost, but the offline optimal cost is zero, contradicting (\ref{competitive-measurement-feedback-cond}). Now suppose that the controller always sets $u = 0$, irrespective of $y_0$.  In this case, the competitive ratio is simply $$\sup_{w \in \ell_2} \frac{\left\| Gw\right\|^2}{\|\Delta^{-1}Gw\|^2},$$ where  $\Delta$ is defined in Lemma \ref{canonical-factorization-lemma}. Let $w_2 = \Delta^{-1}Gw$. The competitive ratio becomes 
\begin{eqnarray*}
\sup_{w_2 \in \ell_2} \frac{\left\| \Delta w_2\right\|^2}{\|w_2\|^2} &=& \sup_{w_2 \in \ell_2} \frac{\left\| \Delta^* w_2\right\|^2}{\|w_2\|^2}\\
&=& \sup_{w_2 \in \ell_2} \frac{w_2^*(I + FF^*)w_2}{\|w_2\|^2} \\
&=& 1 + \sup_{w_2 \in \ell_2} \frac{w_2^*FF^*w_2}{\|w_2\|^2}\\
&=& 1 + \|F\|^2.
\end{eqnarray*}
It is easy to check that $F$ is bounded if and only if $A$ is stable.
\end{proof}

We next prove that there is no controller whose regret bounded by the joint pathlength of the driving disturbance and the measurement disturbance; this non-existence result holds for all linear systems, stable or unstable. The key observation is that the regret bound implies that the online controller must incur zero cost whenever $w$ and $v$ are constant. In the measurement-feedback setting the controller is unable to directly observe $w$ and $v$, and hence must set $u = 0$ at all times. However, there exist choices of $w$ and $v$ with zero pathlength such that this controller incurs positive regret, contradicting the pathlength bound.

\begin{theorem}
Fix $\gamma > 0$. There does not exist a causal controller $\pi$ such that 
\begin{equation} \label{pathlength-w-pathlength-v-cond}
J(\pi, w, v) - J(\pi_0, w) < \gamma^2\left( \|D_p w\|_2^2 + \|D_r v\|_2^2 \right)
\end{equation}
for all driving disturbances $w$ and and all measurement disturbances $v$.
\end{theorem} 
\begin{proof}

Suppose there exists  a measurement $y_0$ such that $\pi(y_0) \neq 0$. Set $w = 0$, $v = y_0$. The online controller incurs positive cost, but the offline optimal cost is zero, hence the online controller incurs positive regret. This contradicts the fact that the right-hand side of (\ref{pathlength-w-pathlength-v-cond}) is zero. Now suppose that the controller always sets $u = 0$, irrespective of $y_0$. The left-hand side of (\ref{pathlength-w-pathlength-v-cond}) is 
\begin{align*} 
\left\| Gw\right\|^2 - \|\Delta^{-1}Gw\|^2 &= w^*G^*(I - (I + FF^*)^{-1})Gw \\
&= w^*G^*F(I + F^*F)^{-1}F^*Gw,
\end{align*}
where  $\Delta$ is defined in Lemma \ref{canonical-factorization-lemma}. Set $w$ to be a nonzero constant sequence such that $F^*Gw \neq 0$. Set $v = 0$. Then the left-hand side of (\ref{pathlength-w-pathlength-v-cond}) is strictly positive, but the right-hand side is zero for all $\gamma$.
\end{proof}

We now prove that in unstable systems there is no controller whose regret is bounded by the energy of the driving disturbance and the pathlength of the measurement disturbance; in stable systems the only such controller is the controller which always sets $u = 0$.
The key observation is that the regret bound implies that the online controller must incur zero cost whenever $w = 0$ and $v$ is constant. In the measurement-feedback setting the controller is unable to directly observe $w$ and $v$, and hence must set $u = 0$ at all times.

\begin{theorem}
Fix $\gamma > 0$ and suppose $A$ is unstable. There does not exist a causal controller $\pi$ such that
\begin{equation} \label{energy-w-pathlength-v-cond}
J(\pi, w, v) - J(\pi_0, w) < \gamma^2\left( \|w\|_2^2 + \|D_r v\|_2^2 \right)
\end{equation}
for all driving disturbances $w$ and and all measurement disturbances $v$. If $A$ is stable, then the only controller which satisfies (\ref{energy-w-pathlength-v-cond}) for any value of $\gamma$ is the ``zero controller'' which always sets $u = 0$. This controller satisfies (\ref{energy-w-pathlength-v-cond}) with $\gamma = \|\Delta_2^{-*}F^*G\|$, where $\Delta_2$ is the unique causal and causally invertible operator such that $\Delta_2^*\Delta_2 = I + F^*F$.
\end{theorem} 

\begin{proof}
Suppose there exists  a measurement $y_0$ such that $\pi(y_0) \neq 0$. Set $w = 0$, $v = y_0$. The online controller incurs positive cost, but the offline optimal cost is zero, hence the online controller incurs positive regret. This contradicts the fact that the right-hand side of (\ref{energy-w-pathlength-v-cond}) is zero. Now suppose that the controller always sets $u = 0$, irrespective of $y_0$. The left-hand side of (\ref{energy-w-pathlength-v-cond}) is 
\begin{align*}
\left\| Gw\right\|^2 - \|\Delta^{-1}Gw\|^2 &= w^*G^*(I - (I + FF^*)^{-1})Gw \\
&= w^*G^*F(I + F^*F)^{-1}F^*Gw,
\end{align*}  where  $\Delta$ is defined in Lemma \ref{canonical-factorization-lemma}. The smallest possible value of $\gamma^2$ such that (\ref{energy-w-pathlength-v-cond}) holds is given by $$ \sup_{w \in \ell_2} \frac{w^*G^*F(I + F^*F)^{-1}F^*Gw}{\|w\|^2} = \|\Delta_2^{-*}F^*G\|^2, $$
where $\Delta_2$ is the unique causal and causally invertible operator such that $\Delta_2^*\Delta_2 = I + F^*F$. We note that the operator $\Delta_2^{-*}F^*G$ is bounded if and only if $A$ is stable.
\end{proof}

\section{Regret bounded by the joint energy of $w$ and $v$} \label{energy-sec}

In this section we describe a causal controller $\pi$ whose regret is bounded by the joint energy of the driving disturbance and the measurement disturbance. This presents a sharp contrast with the negative results of the previous sections; the key difference is that the only pair of disturbances $(w, v)$ whose joint energy is zero is simply $w = 0, v = 0$. It is easy to guarantee that $K$  sets $u  = 0$ on this specific instance without also requiring that $K$ sets $u = 0$ on all other instances. We derive the measurement-feedback controller whose regret has optimal dependence on the joint energy of the driving disturbance and the measurement disturbance via a reduction to $H_{\infty}$ measurement-feedback control.

\begin{theorem}
Fix $\gamma > 0$ and define $\widehat{A}, \widehat{B}_u, \widehat{B}_w, \widehat{C}, \widehat{L}$ as in 
(\ref{energy-w-energy-v-measurement-feedback-new-dynamics}). There exists a causal measurement-feedback controller $\pi$ such that
\begin{equation} \label{energy-w-energy-v-regret-cond}
J(\pi, w, v) - J(\pi_0, w)  < \gamma^2(\|w\|_2^2 + \| v\|_2^2)    
\end{equation}
if and only if the control DARE $$P_c = \widehat{A}^*P_c\widehat{A} + \widehat{L}^*\widehat{L} - K_c^* R_c K_c, $$ and the estimation DARE $$P_e = \widehat{A}P_e \widehat{A}^* +    \widehat{B}_w \widehat{B}_w^* - K_e R_e K_e^*,$$ where we define 
$$\begin{cases}
K_c = R_c^{-1}\begin{bmatrix} \widehat{B}_u^* \\ \widehat{B}_w^* \end{bmatrix}P_c \widehat{A} \\
\\
R_c = \begin{bmatrix} I_m & 0 \\ 0 & -I_p \end{bmatrix} + \begin{bmatrix} \widehat{B}_u^* \\    \widehat{B}_w^* \end{bmatrix}P_c \begin{bmatrix} \widehat{B}_u &    \widehat{B}_w \end{bmatrix}, \\
\\
K_e = \widehat{A}P^d \begin{bmatrix}  \widehat{C} & L \end{bmatrix}R_e^{-1}, \\
\\
R_e = \begin{bmatrix} I_r & 0 \\ 0 & -I_n \end{bmatrix} + \begin{bmatrix}  \widehat{C} \\ \widehat{L} \end{bmatrix}P_c \begin{bmatrix}  \widehat{C}^* & \widehat{L}^* \end{bmatrix}, 
\end{cases} $$  have solutions $P_c \succeq 0$ and $P_e \succeq 0$ such that 
\begin{enumerate}
    \item The matrix $\widehat{A} - \begin{bmatrix} \widehat{B}_u &    \widehat{B}_w \end{bmatrix}K_c$ is stable.
    \item The matrix $R_c$ has $m$ positive eigenvalues and $p$ negative eigenvalues.
    \item The matrix $\widehat{A} - K_e\begin{bmatrix}  \widehat{C} \\ \widehat{L} \end{bmatrix}$ is stable.
    \item The matrix $R_e$ has $r$ positive eigenvalues and $n$ negative eigenvalues.
    \item $\rho(P_c P_e) < 1$.
\end{enumerate}
If these conditions are satisfied, then one possible choice of $\pi$ is given by $$u_t = - K_u (\xi_t + P \widehat{C}^*(I_r +  \widehat{C} P \widehat{C}^*)^{-1}(y_t - \widehat{C} \xi_t)),$$ where the synthetic state $\xi \in \mathbb{R}^{2n}$ evolves according to the linear dynamics equation $$\xi_{t+1} = (\widehat{A} - \widehat{B}_wK_w)(\xi_t + P \widehat{C}^*(I_r +  \widehat{C} P \widehat{C}^*)^{-1}(y_t - \widehat{C} \xi_t)) + \widehat{B}_u u_t.$$ The matrices $K_u \in \mathbb{R}^{m \times 2n}$ and $K_w \in \mathbb{R}^{p \times 2n}$ are defined as $$\begin{bmatrix} K_u \\ K_w \end{bmatrix} = K_c,$$ and we define $P \in \mathbb{R}^{2n \times 2n}$ as $$ P = P_e(I_n - P_c P_e)^{-1}. $$
\end{theorem}


We note that the regret-optimal controller can be easily obtained via bisection on $\gamma$.

\begin{proof}
The regret condition (\ref{energy-w-energy-v-regret-cond}) can be rewritten as 
\begin{equation} \label{nergy-w-energy-v-regret-cond-2}
 T_K^* T_K  \preceq \left(\gamma^2\begin{bmatrix} I_p & 0 \\ 0 & I_r \end{bmatrix} + T_{K_0}^* T_{K_0}\right).
\end{equation}
Let $\Delta_2$ be the unique causal and causally invertible operator such that \begin{equation} \label{energy-w-energy-v-factorization}
\gamma^2 I_p + G^*\left(I + FF^*\right)^{-1}G = \Delta_2^* \Delta_2.
\end{equation} 
Then $$\gamma^2\begin{bmatrix} I_p & 0 \\ 0 & I_r \end{bmatrix} + T_{K_0}^* T_{K_0}  = \begin{bmatrix} \Delta_2 & 0 \\ 0 & \gamma I \end{bmatrix}^* \begin{bmatrix} \Delta_2 & 0 \\ 0 & \gamma I \end{bmatrix}. $$
Define $$ \begin{bmatrix} \widehat{w} \\ \widehat{v} \end{bmatrix} = \begin{bmatrix} \Delta_2 & 0 \\ 0 & \gamma I \end{bmatrix} \begin{bmatrix} w \\ v \end{bmatrix}. $$
Condition (\ref{nergy-w-energy-v-regret-cond-2}) can be rewritten as $$\begin{bmatrix} \widehat{w} \\ \widehat{v} \end{bmatrix}^* T_{\widehat{K}}^*T_{\widehat{K}} \begin{bmatrix} \widehat{w} \\ \widehat{v} \end{bmatrix} < \left\|\begin{bmatrix} \widehat{w} \\ \widehat{v} \end{bmatrix}\right\|_2^2,$$ or equivalently as $$\|T_{\widehat{K}}\| < 1,$$ where we define
 $$T_{\widehat{K}} = T_K\begin{bmatrix} \Delta_2 & 0 \\ 0 & \gamma I \end{bmatrix}^{-1}.$$ Using the parameterization (\ref{transfer-operator}) of $T_K$, we see that 
$$ T_{\widehat{K}} = \begin{bmatrix} G\Delta_2^{-1} & 0 \\ 0 & 0 \end{bmatrix}  + \begin{bmatrix}F \\ I \end{bmatrix} (\gamma^{-1} Q) \begin{bmatrix} \gamma J\Delta_2^{-1} & I \end{bmatrix}.
$$
Notice that $T_{\widehat{K}}$ itself has the form of a transfer operator described in (\ref{transfer-operator}); it is the transfer operator with Youla parameter $\widehat{Q} = \gamma^{-1} Q$ in the system 
\begin{equation} \label{energy-w-energy-v-regret-system}
s = \widehat{F}u + \widehat{G}\widehat{w}, \hspace{3mm} \widehat{y} = \widehat{H}u + \widehat{J}\widehat{w} + \widehat{v},
\end{equation}
where we define $\widehat{F} = F, \widehat{G} =  G\Delta_2^{-1}, \widehat{H} = \gamma H, \widehat{J} = \gamma J \Delta_2^{-1}$. It is now clear that a controller $K$  satisfying (\ref{energy-w-energy-v-regret-cond}) exists if and only if there exists a controller $\widehat{K}$ in the system (\ref{energy-w-energy-v-regret-system}) such that $\|T_{\widehat{K}} \| < 1 $. If such a controller $\widehat{K}$ exists, then we can easily recover $K$ from $\widehat{K}$ by setting $K = \gamma \widehat{K}$; notice that this is the unique choice of $K$ which is consistent with the relations $\widehat{H} = \gamma H, \widehat{Q} = \gamma^{-1} Q $.

In order to assign state-space structure to $\widehat{F}, \widehat{G}, \widehat{H}, \widehat{J}$, we must first find $\Delta_2(z)$. Let $\Delta$ be the unique causal and causally invertible operator such that $$I + FF^* = \Delta\Delta^*.$$ Lemma \ref{canonical-factorization-lemma} shows that $$\Delta(z) = (I + L(zI - A)^{-1}K )\Sigma^{1/2}  .$$ It follows that $\Delta^{-1}(z)G(z)$ is given by $$\Delta^{-1}(z)G(z) = \Sigma^{-1/2} L(zI - (A - KL))^{-1}B_w.$$
Define $\widetilde{A} = A - KL $.
Notice that $ \gamma^2 I_p +  G^*\Delta^* \Delta^{-1}G$ can be rewritten as
$$ \begin{bmatrix} B_w^*(z^{-*}I_n - \widetilde{A})^{-*} & I_p  \end{bmatrix} \begin{bmatrix} L^*\Sigma^{-1}L & 0 \\ 0 & \gamma^2 I_p \end{bmatrix} \begin{bmatrix} (zI_n - \widetilde{A})^{-1}B_w \\ I_p \end{bmatrix}.$$
Applying Lemma \ref{equivalence-class-lemma-1}, we see that this equals $$\begin{bmatrix} B_w^*(z^{-*}I_n - \widetilde{A})^{-*} & I \end{bmatrix} \Lambda_2(P_2) \begin{bmatrix} (zI_n - \widetilde{A})^{-1}B_w \\ I \end{bmatrix},$$
where $P_2$ is an arbitrary Hermitian matrix and we define $$\Lambda_2(P_2) = \begin{bmatrix} L^*\Sigma^{-1}L - P_2 +  \widetilde{A}^*P_2\widetilde{A} &  \widetilde{A}^*P_2B_w \\  B_w^*P_2\widetilde{A} & \gamma^2I + B_w^*P_2B_w \end{bmatrix}.$$ Notice that the $\Lambda_2(P_2)$ can be factored as $$\begin{bmatrix} I & K_2^*(P_2) \\ 0 & I \end{bmatrix} \begin{bmatrix} \Gamma_2(P_2) & 0 \\ 0 & \Sigma_2(P_2) \end{bmatrix} \begin{bmatrix} I & 0 \\ K_2(P_2) & I \end{bmatrix}, $$
where we define $$\Gamma_2(P_2) =  L^*\Sigma^{-1}L - P_2 +  \widetilde{A}^*P_2\widetilde{A} - K_2^*(P_2)\Sigma_2K_2(P_2),$$ 
$$K_2(P_2) = \Sigma_2^{-1}(P_2) B_w^*P_2\widetilde{A}, \hspace{3mm} \Sigma_2(P_2) = \gamma^2I + B_w^*P_2B_w.$$
It is clear that $(\widetilde{A}, B_w)$ is stabilizable (in fact, $\widetilde{A}$ is stable), therefore the Riccati equation $\Gamma_2(P_2) = 0$ has a unique stabilizing solution (Theorem E.6.2 in \cite{kailath2000linear}).  Suppose $P_2$ is chosen to be this solution, and define $K_2 = K_2(P_2)$, $\Sigma_2 = \Sigma_2(P_2)$. We immediately obtain the factorization (\ref{energy-w-energy-v-factorization}), where we define 
\begin{equation} \label{energy-u-energy-v-delta}
\Delta_2(z) = \Sigma_2^{1/2}(I + K_2(zI - \widetilde{A})^{-1}B_w).
\end{equation}
We have $$\Delta_2^{-1}(z) = (I - K_2(zI - (\widetilde{A} - B_wK_2))^{-1}B_w) \Sigma_2^{-1/2}. $$ We note that $ \widetilde{A} - B_wK_2$ is stable and hence $\Delta^{-1}(z)$ is causal and bounded since its poles are strictly contained in the unit circle. Define 
\begin{equation} \label{energy-w-energy-v-measurement-feedback-new-dynamics}
\begin{cases}
\widehat{A} =  \begin{bmatrix} A & -B_wK_2  \\ 0 & \widetilde{A} - B_wK_2 \end{bmatrix},\\
\\
\widehat{B}_w = \begin{bmatrix} B_w \Sigma_2^{-1/2} \\ B_w \Sigma_2^{-1/2} \end{bmatrix}, \\ 
\\
\widehat{B}_u = \begin{bmatrix} B_u \\ 0 \end{bmatrix},\\
\\
\widehat{C} =  \begin{bmatrix} \gamma C & 0 \end{bmatrix}, \\
\\
\widehat{L} =  \begin{bmatrix} L & 0 \end{bmatrix}.
\end{cases}
\end{equation}
Recall that $\widehat{F} = F, \widehat{G} =  G\Delta_2^{-1}, \widehat{H} = \gamma H, \widehat{J} = \gamma J \Delta_2^{-1}$; it follows that $\widehat{F}, \widehat{G}, \widehat{H}, \widehat{J}$ are given by
$$\widehat{F}(z) =  \widehat{L}(zI - \widehat{A})^{-1}\widehat{B}_u, \hspace{5mm} \widehat{G}(z) = \widehat{L}(zI - \widehat{A})\widehat{B}_w, $$
$$\widehat{H}(z) = \widehat{C}(zI - \widehat{A})^{-1}\widehat{B}_u, \hspace{5mm} \widehat{J}(z) = \widehat{C}(zI - \widehat{A})\widehat{B}_w.$$
\end{proof}


\section{Regret bounded by the pathlength of $w$ and the energy of $v$} \label{pathlength-sec}

In this section we describe a causal controller $\pi$ whose regret is bounded by the pathlength of the driving disturbance and the energy of the measurement disturbance. This presents a sharp contrast with the negative results of the previous sections; the key difference is that the only pairs of disturbances $(w, v)$ such that the pathlength of the $w$ is zero and the energy of $v$ is zero are pairs where $w$ is constant and $v$ is zero. It is easy to guarantee that the causal controller $K$ matches the noncausal  controller $K_0$ on these specific instances, without constraining the behavior of $K$ on all other instances. We derive the measurement-feedback controller whose regret has optimal dependence on the pathlength of the driving disturbance and the energy of the measurement disturbance via a reduction to $H_{\infty}$ measurement-feedback control.

\begin{theorem} \label{pathlength-w-energy-v-measurement-feedback-thm}
Fix $\gamma > 0$ and define $\widehat{A}, \widehat{B}_u, \widehat{B}_w, \widehat{C}, \widehat{L}$ as in 
(\ref{pathlength-w-energy-v-measurement-feedback-new-dynamics}). There exists a causal measurement-feedback controller $\pi$ such that
\begin{equation} \label{pathlength-w-energy-v-regret-cond}
J(\pi, w, v) - J(\pi_0, w)  < \gamma^2(\|D_p w\|_2^2 + \| v\|_2^2)   
\end{equation}
if and only if the control DARE $$P_c = \widehat{A}^*P_c\widehat{A} + \widehat{L}^*\widehat{L} - K_c^* R_c K_c, $$ and the estimation DARE $$P_e = \widehat{A}P_e \widehat{A}^* +    \widehat{B}_w \widehat{B}_w^* - K_e R_e K_e^*,$$ where we define 
$$\begin{cases}
K_c = R_c^{-1}\begin{bmatrix} \widehat{B}_u^* \\ 
\widehat{B}_w^* \end{bmatrix}P_c \widehat{A} \\
\\
R_c = \begin{bmatrix} I_m & 0 \\ 0 & -I_p \end{bmatrix} + \begin{bmatrix} \widehat{B}_u^* \\    \widehat{B}_w^* \end{bmatrix}P_c \begin{bmatrix} \widehat{B}_u &    \widehat{B}_w \end{bmatrix}, \\
\\
K_e = \widehat{A}P^d \begin{bmatrix}  \widehat{C} & L \end{bmatrix}R_e^{-1}, \\
\\
R_e = \begin{bmatrix} I_r & 0 \\ 0 & -I_n \end{bmatrix} + \begin{bmatrix}  \widehat{C} \\ \widehat{L} \end{bmatrix}P_c \begin{bmatrix}  \widehat{C}^* & \widehat{L}^* \end{bmatrix}, 
\end{cases} $$  have solutions $P_c \succeq 0$ and $P_e \succeq 0$ such that 
\begin{enumerate}
    \item The matrix $\widehat{A} - \begin{bmatrix} \widehat{B}_u &    \widehat{B}_w \end{bmatrix}K_c$ is stable.
    \item The matrix $R_c$ has $m$ positive eigenvalues and $p$ negative eigenvalues.
    \item The matrix $\widehat{A} - K_e\begin{bmatrix}  \widehat{C} \\ \widehat{L} \end{bmatrix}$ is stable.
    \item The matrix $R_e$ has $r$ positive eigenvalues and $n$ negative eigenvalues.
    \item $\rho(P_c P_e) < 1$.
\end{enumerate}
If these conditions are satisfied, then one possible choice of $\pi$ is given by $$u_t = - K_u (\xi_t + P \widehat{C}^*(I_r +  \widehat{C} P \widehat{C}^*)^{-1}(y_t - \widehat{C} \xi_t)),$$ where the synthetic state $\xi \in \mathbb{R}^{2n + p}$ evolves according to the linear dynamics equation $$\xi_{t+1} = (\widehat{A} - \widehat{B}_wK_w)(\xi_t + P \widehat{C}^*(I_r +  \widehat{C} P \widehat{C}^*)^{-1}(y_t - \widehat{C} \xi_t)) + \widehat{B}_u u_t.$$ The matrices $K_u \in \mathbb{R}^{m \times (2n+p)}$ and $K_w \in \mathbb{R}^{p \times (2n+p)}$ are defined as $$\begin{bmatrix} K_u \\ K_w \end{bmatrix} = K_c,$$ and we define $P \in \mathbb{R}^{(2n + p)\times (2n + p)}$ as $$ P = P_e(I_n - P_c P_e)^{-1}. $$
\end{theorem}

We note that the regret-optimal controller can be easily obtained via bisection on $\gamma$.

\begin{proof}
The regret condition (\ref{pathlength-w-energy-v-measurement-feedback-thm}) can be rewritten in matrix form as 
\begin{equation} \label{pathlength-w-energy-v-measurement-feedback-cond-2}
\begin{bmatrix} w \\ v \end{bmatrix}^* T_K^* T_K \begin{bmatrix} w \\ v \end{bmatrix} < \begin{bmatrix} w \\ v \end{bmatrix}^* \left(\gamma^2\begin{bmatrix} D_p^*D_p & 0 \\ 0 & I \end{bmatrix} + T_{K_0}^* T_{K_0}\right)\begin{bmatrix} w \\ v \end{bmatrix}.
\end{equation}
Let $\Delta_2$ be the unique causal and causally invertible operator such that \begin{equation} \label{pathlength-control-factorization}
\gamma^2 D_p^*D_p + G^*\left(I + FF^*\right)^{-1}G = \Delta_2^* \Delta_2.
\end{equation} 
Then $$\gamma^2\begin{bmatrix} D_p^*D_p & 0 \\ 0 & I \end{bmatrix} + T_{K_0}^* T_{K_0}  = \begin{bmatrix} \Delta_2 & 0 \\ 0 & \gamma I \end{bmatrix}^* \begin{bmatrix} \Delta_2 & 0 \\ 0 & \gamma I \end{bmatrix}. $$
Define $$ \begin{bmatrix} \widehat{w} \\ \widehat{v} \end{bmatrix} = \begin{bmatrix} \Delta_2 & 0 \\ 0 & \gamma I \end{bmatrix} \begin{bmatrix} w \\ v \end{bmatrix}.$$
Condition (\ref{pathlength-w-energy-v-measurement-feedback-cond-2}) can be rewritten as $$\begin{bmatrix} \widehat{w} \\ \widehat{v} \end{bmatrix}^* T_{\widehat{K}}^*T_{\widehat{K}} \begin{bmatrix} \widehat{w} \\ \widehat{v} \end{bmatrix} < \left\|\begin{bmatrix} \widehat{w} \\ \widehat{v} \end{bmatrix}\right\|_2^2,$$ or equivalently as $$\|T_{\widehat{K}}\| < 1,$$ where we define
 $$T_{\widehat{K}} = T_K\begin{bmatrix} \Delta_2 & 0 \\ 0 & \gamma I \end{bmatrix}^{-1}.$$ Using the parameterization (\ref{transfer-operator}) of $T_K$, we see that 
$$ T_{\widehat{K}} = \begin{bmatrix} G\Delta_2^{-1} & 0 \\ 0 & 0 \end{bmatrix}  + \begin{bmatrix}F \\ I \end{bmatrix} (\gamma^{-1} Q) \begin{bmatrix} \gamma J\Delta_2^{-1} & I \end{bmatrix}.
$$
Notice that $T_{\widehat{K}}$ itself has the form of a transfer operator described in (\ref{transfer-operator}); it is the transfer operator with Youla parameter $\widehat{Q} = \gamma^{-1} Q$ in the system 
\begin{equation} \label{pathlength-w-energy-v-measurement-feedback-system}
s = \widehat{F}u + \widehat{G}\widehat{w}, \hspace{3mm} \widehat{y} = \widehat{H}\widehat{u} + \widehat{J}\widehat{w} + \widehat{v},
\end{equation}
where we define $\widehat{F} = F, \widehat{G} =  G\Delta_2^{-1}, \widehat{H} = \gamma H, \widehat{J} = \gamma J \Delta_2^{-1}$. It is now clear that a controller $K$  satisfying (\ref{pathlength-w-energy-v-measurement-feedback-cond-2}) exists if and only if there exists a controller $\widehat{K}$ in the system (\ref{pathlength-w-energy-v-measurement-feedback-system}) such that $\|T_{\widehat{K}} \| < 1 $. If such a controller $\widehat{K}$ exists, then we can easily recover $K$ from $\widehat{K}$ by setting $K = \gamma \widehat{K}$; notice that this is the unique choice of $K$ which is consistent with the relations $\widehat{H} = \gamma H, \widehat{Q} = \gamma^{-1} Q $.

In order to assign state-space structure to $\widehat{F}, \widehat{G}, \widehat{H}, \widehat{J}$, we must first find $\Delta_2(z)$. Let $\Delta$ be the unique causal and causally invertible operator such that $$I + FF^* = \Delta\Delta^*.$$ Lemma \ref{canonical-factorization-lemma} shows that $$\Delta(z) = (I + L(zI - A)^{-1}K )\Sigma^{1/2}  .$$ It follows that $\Delta^{-1}(z)G(z)$ is given by $$\Delta^{-1}(z)G(z) = \Sigma^{-1/2} L(zI - (A - KL))^{-1}B_w.$$
Define 
\begin{equation*} 
 \widetilde{L} = \begin{bmatrix} \Sigma^{-1/2}L & 0 \\ 0 & \gamma I_p \end{bmatrix}, \hspace{3mm} \widetilde{S} = \begin{bmatrix} 0 \\ \gamma^2 I_p \end{bmatrix}, 
\end{equation*}
\begin{equation} \label{tilde-A-tilde-B-def}
 \widetilde{A} = \begin{bmatrix} A - KL & 0 \\ 0 & 0 \end{bmatrix}, \hspace{3mm} \widetilde{B}_w = \begin{bmatrix} B_w \\  -I_p \end{bmatrix}.
\end{equation}
Notice that we can rewrite
$$ \begin{bmatrix} \gamma^2 D_p^*(z^{-*})D_p(z) +  G^*\Delta^{-*} \Delta^{-1}G & 0 \\ 0 & \gamma^2I \end{bmatrix}$$ 
as
$$ \begin{bmatrix} \widetilde{B}_w^*(z^{-*}I - \widetilde{A})^{-*} & I  \end{bmatrix} \begin{bmatrix} \widetilde{L}^*\widetilde{L} & \widetilde{S} \\ \widetilde{S}^* & \gamma^2 I \end{bmatrix} \begin{bmatrix} (zI - \widetilde{A})^{-1}\widetilde{B}_w \\ I \end{bmatrix}.$$
Applying Lemma \ref{equivalence-class-lemma-1}, we see that this equals $$\begin{bmatrix} \widetilde{B}_w^*(z^{-*}I - \widetilde{A})^{-*} & I \end{bmatrix} \Lambda_2(P_2) \begin{bmatrix} (zI - \widetilde{A})^{-1}\widetilde{B}_w \\ I \end{bmatrix},$$
where $P_2$ is an arbitrary Hermitian matrix and we define $$\Lambda_2(P_2) = \begin{bmatrix} \widetilde{L}^*\widetilde{L} - P_2 +  \widetilde{A}^*P_2\widetilde{A} & \widetilde{S} + \widetilde{A}^*P_2\widetilde{B}_w \\ \widetilde{S}^* + \widetilde{B}_w^*P_2\widetilde{A} & \gamma^2I + \widetilde{B}_w^*P_2\widetilde{B}_w \end{bmatrix}.$$ Notice that the $\Lambda_2(P_2)$ can be factored as $$\begin{bmatrix} I & K_2^*(P_2) \\ 0 & I \end{bmatrix} \begin{bmatrix} \Gamma_2(P_2) & 0 \\ 0 & \Sigma_2(P_2) \end{bmatrix} \begin{bmatrix} I & 0 \\ K_2(P_2) & I \end{bmatrix}, $$
where we define $$\Gamma_2(P_2) = \widetilde{L}^*\widetilde{L} - P_2 +  \widetilde{A}^*P_2\widetilde{A} - K_2^*(P_2)\Sigma_2K_2(P_2),$$ 
\begin{align*} 
K_2(P_2) &= \Sigma_2^{-1}(P_2)(\widetilde{S}^* + \widetilde{B}_w^*P_2\widetilde{A}) \\
\Sigma_2(P_2) &= \gamma^2I + \widetilde{B}_w^*P_2\widetilde{B}_w.
\end{align*}
It is clear that $(\widetilde{A}, \widetilde{B}_w)$ is stabilizable (this follows from the stability of $A - KL$), therefore the Riccati equation $\Gamma_2(P_2) = 0$ has a unique stabilizing solution (Theorem E.6.2 in \cite{kailath2000linear}).  Suppose $P_2$ is chosen to be this solution, and define $K_2 = K_2(P_2)$, $\Sigma_2 = \Sigma_2(P_2)$. We immediately obtain the factorization (\ref{pathlength-control-factorization}), where we define 
\begin{equation} \label{energy-u-pathlength-v-delta}
\Delta_2(z) = \Sigma_2^{1/2}(I + K_2(zI - \widetilde{A})^{-1}\widetilde{B}_w).
\end{equation}
We have $$\Delta_2^{-1}(z) = (I - K_2(zI - (\widetilde{A} - \widetilde{B}_wK_2))^{-1}\widetilde{B}_w) \Sigma_2^{-1/2}. $$ We note that $ \widetilde{A} - \widetilde{B}_wK_2$ is stable and hence $\Delta_2^{-1}(z)$ is causal and bounded since its poles are strictly contained in the unit circle. Define 
\begin{equation} \label{pathlength-w-energy-v-measurement-feedback-new-dynamics}
\begin{cases}
\widehat{A} =  \begin{bmatrix} A & -B_wK_2  \\ 0 & \widetilde{A} - \widetilde{B}_wK_2  \end{bmatrix} \\ 
\\
\widehat{B}_u = \begin{bmatrix} B_u \\ 0 \end{bmatrix}, \\
\\
\widehat{B}_w = \begin{bmatrix} B_w \Sigma_2^{-1/2} \\ \widetilde{B}_w \Sigma_2^{-1/2} \end{bmatrix}, \\
\\
\widehat{C} =  \begin{bmatrix} \gamma C & 0 \end{bmatrix},\\
\\
\widehat{L} =  \begin{bmatrix} L & 0 \end{bmatrix}.
\end{cases}
\end{equation}
Recall that $\widehat{F} = F, \widehat{G} =  G\Delta_2^{-1}, \widehat{H} = \gamma H, \widehat{J} = \gamma J \Delta_2^{-1}$; it follows that $\widehat{F}, \widehat{G}, \widehat{H}, \widehat{J}$ are given by
$$ \widehat{F}(z) =  \widehat{L}(zI - \widehat{A})^{-1}\widehat{B}_u, \hspace{5mm} \widehat{G}(z) = \widehat{L}(zI - \widehat{A})\widehat{B}_w, $$
$$\widehat{H}(z) = \widehat{C}(zI - \widehat{A})^{-1}\widehat{B}_u,  \hspace{5mm} \widehat{J}(z) = \widehat{C}(zI - \widehat{A})\widehat{B}_w.$$
\end{proof}

\section{Numerical Experiments}
We evaluate the performance of our regret-optimal control algorithms in a double integrator system. The double integrator is a simple dynamical system that models the one-dimensional kinematics of a moving object with position $x$ and velocity $\dot{x}$; the goal of the controller is to stabilize the object by keeping $(x, \dot{x})$ as close to $(0, 0)$ as possible. The discrete-time dynamics are $$ \begin{bmatrix} x_{t+1} \\ \dot{x}_{t+1} \end{bmatrix} = \begin{bmatrix} 1 & \delta_t \\ 0 & 1\end{bmatrix} \begin{bmatrix} x_t \\ \dot{x}_t\end{bmatrix} + \begin{bmatrix} 0 \\ \delta_t \end{bmatrix}u_t + \begin{bmatrix} 0 \\ \delta_t \end{bmatrix} w_t, $$ where $\delta_t$ is the discretization parameter; in our experiments we set $\delta_t = 0.1$ seconds. In our experiments we take $Q = I_2$ and initialize $x$ and $\dot{x}$ to zero. We study the relative performance of the $H_2$-optimal, $H_{\infty}$-optimal, energy-optimal, and pathlength-optimal controllers across various kinds of driving disturbances; in all of our experiments we take the measurement disturbance $v$ to be sampled i.i.d. from $\mathcal{N}(0, 1)$ in each timestep.

We first consider an optimistic setting where both the driving disturbance $w$ and the measurement disturbance $v$ are drawn i.i.d. from $\mathcal{N}(0, 1)$ in each timestep (Figure \ref{double_integrator_measurement_feedback_w_gaussian_fig}). As expected, the $H_2$ controller performs the best out of all causal controllers, with the energy-optimal controller in second place. The poor performance of the pathlength-optimal controller is explained by the fact that a Gaussian driving disturbance has high pathlength relative to its energy. 

We next consider an adversarial setting where the driving disturbance is an impulse, i.e. the disturbance is zero across the whole time horizon except for a spike near $t = 500$ (Figure \ref{double_integrator_measurement_feedback_w_dirac_fig}). As expected, the $H_{\infty}$ controller performs the best out of all causal controllers, with the energy-optimal controller close behind. The overly optimistic $H_2$ controller performs poorly on this disturbance.

Lastly, we consider a setting where the driving disturbance is a Gaussian random walk, i.e. $w_t$ is the sum of the first $t$ of a set of random variables, each of which is drawn i.i.d from $\mathcal{N}(0, 1)$. Notice that the energy of this disturbance is much higher than its pathlength. As expected, the pathlength-optimal controller now outperforms all other causal controllers and closely tracks the clairvoyant noncausal controller.

\begin{figure}[h!]
\centering
\includegraphics[width=0.5\textwidth]{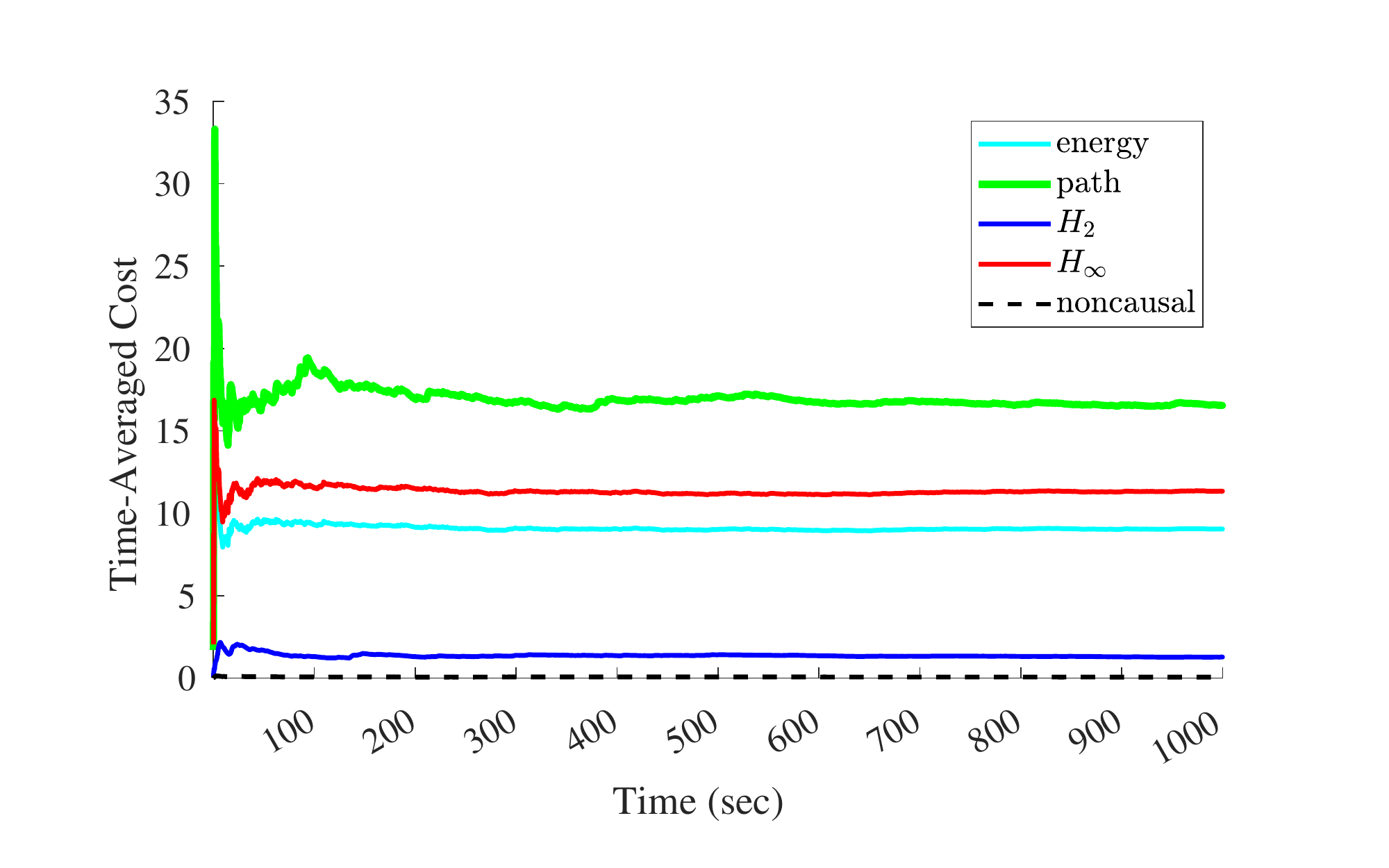}
\caption{Relative performance of measurement-feedback controllers in the double integrator system when both the driving disturbance and the measurement disturbance are drawn i.i.d from $\mathcal{N}(0, 1)$.}
\label{double_integrator_measurement_feedback_w_gaussian_fig}
\end{figure}

\begin{figure}[h!]
\centering
\includegraphics[width=0.5\textwidth]{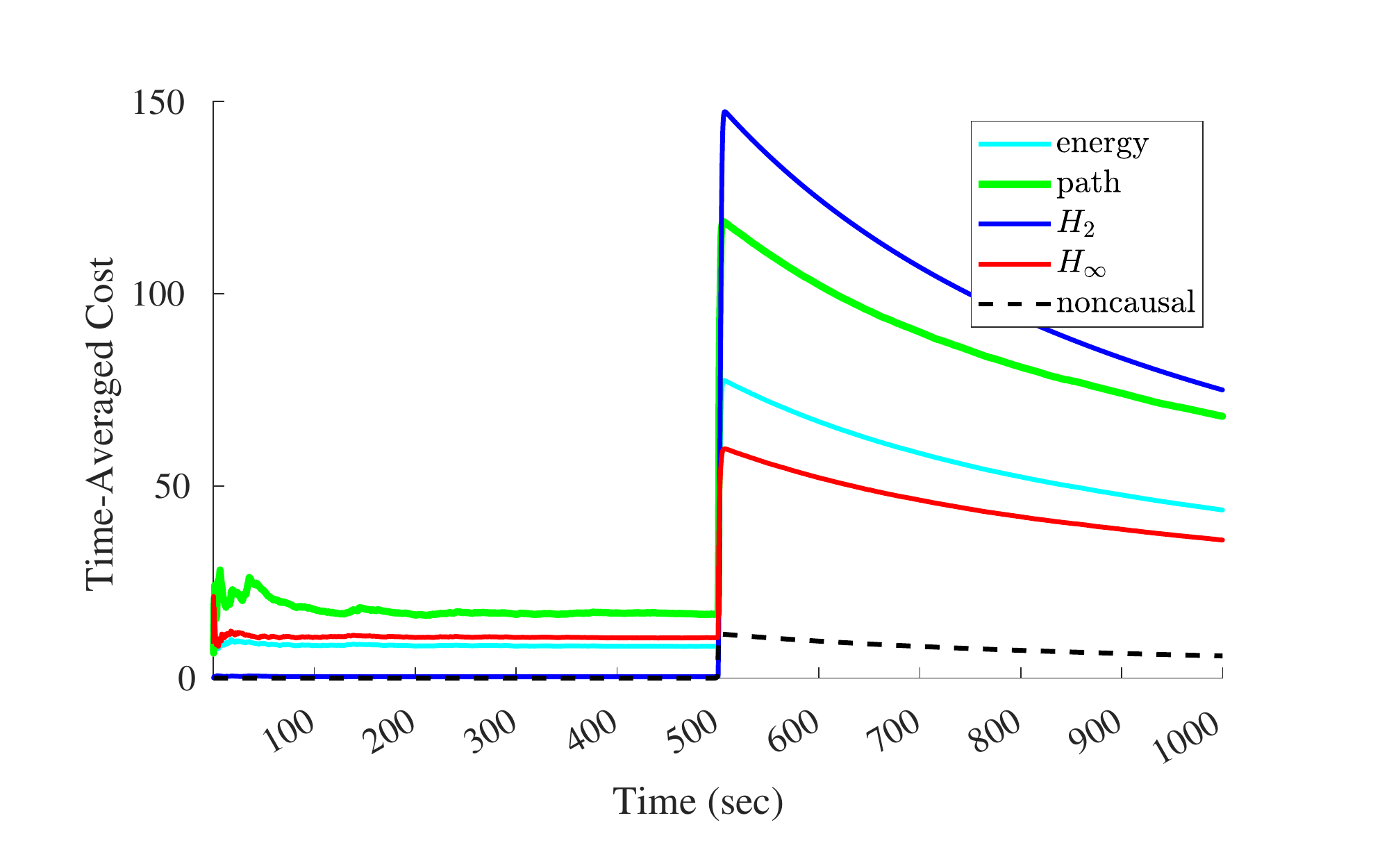}
\caption{Relative performance of measurement-feedback controllers in the double integrator system when the driving disturbance is an impulse which spikes near $t = 500$.}
\label{double_integrator_measurement_feedback_w_dirac_fig}
\end{figure}

\begin{figure}[h!]
\centering
\includegraphics[width=0.5\textwidth]{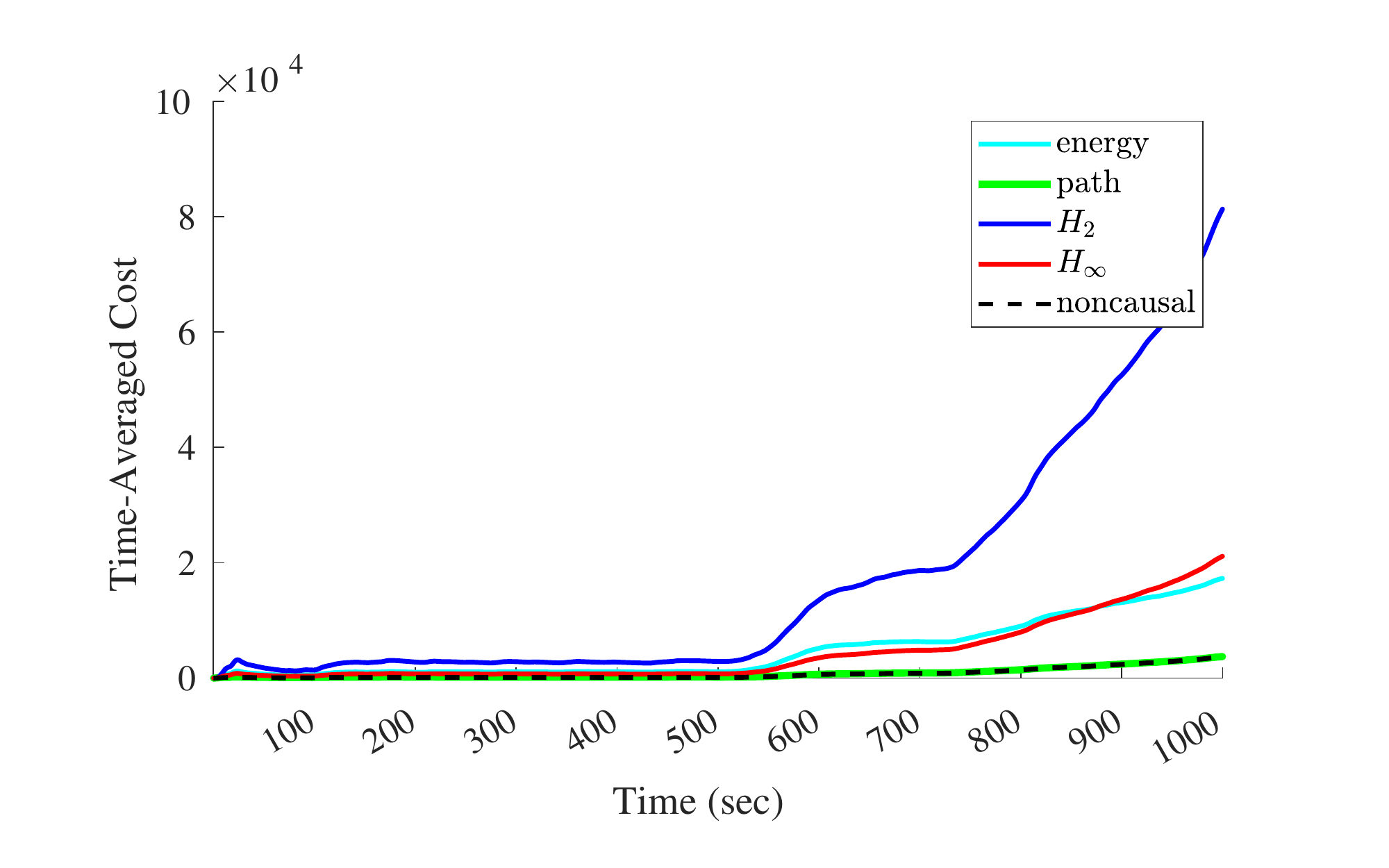}
\caption{Relative performance of measurement-feedback controllers in the double integrator system when the driving disturbance is a Gaussian random walk.}
\label{double_integrator_measurement_feedback_w_gaussian_walk_fig}
\end{figure}

\printbibliography

@article{goel2019online,
  title={An online algorithm for smoothed regression and lqr control},
  author={Goel, Gautam and Wierman, Adam},
  journal={Proceedings of Machine Learning Research},
  volume={89},
  pages={2504--2513},
  year={2019},
  publisher={PMLR}
}

@article{simchowitz2020making,
  title={Making non-stochastic control (almost) as easy as stochastic},
  author={Simchowitz, Max},
  journal={Advances in Neural Information Processing Systems},
  volume={33},
  pages={18318--18329},
  year={2020}
}

@article{shi2020online,
  title={Online optimization with memory and competitive control},
  author={Shi, Guanya and Lin, Yiheng and Chung, Soon-Jo and Yue, Yisong and Wierman, Adam},
  journal={Advances in Neural Information Processing Systems},
  volume={33},
  pages={20636--20647},
  year={2020}
}

@article{goel2021regret,
  title={Regret-Optimal Estimation and Control},
  author={Goel, Gautam and Hassibi, Babak},
  journal={arXiv preprint arXiv:2106.12097},
  year={2021}
}

@inproceedings{goel2021measurement,
  title={Regret-Optimal Measurement-Feedback Control},
  author={Goel, Gautam and Hassibi, Babak},
  booktitle={Learning for Dynamics and Control},
  pages={1270--1280},
  year={2021},
  organization={PMLR}
}

@article{anderson2019system,
  title={System level synthesis},
  author={Anderson, James and Doyle, John C and Low, Steven H and Matni, Nikolai},
  journal={Annual Reviews in Control},
  volume={47},
  pages={364--393},
  year={2019},
  publisher={Elsevier}
}

@article{didier2022system,
  title={A system level approach to regret optimal control},
  author={Didier, Alexandre and Sieber, Jerome and Zeilinger, Melanie N},
  journal={IEEE Control Systems Letters},
  year={2022},
  publisher={IEEE}
}

@inproceedings{sabag2021regret,
  title={Regret-optimal controller for the full-information problem},
  author={Sabag, Oron and Goel, Gautam and Lale, Sahin and Hassibi, Babak},
  booktitle={2021 American Control Conference (ACC)},
  pages={4777--4782},
  year={2021},
  organization={IEEE}
}

@article{agarwal2019online,
  title={Online control with adversarial disturbances},
  author={Agarwal, Naman and Bullins, Brian and Hazan, Elad and Kakade, Sham M and Singh, Karan},
  journal={arXiv preprint arXiv:1902.08721},
  year={2019}
}

@article{goel2020power,
  title={The Power of Linear Controllers in LQR Control},
  author={Goel, Gautam and Hassibi, Babak},
  journal={arXiv preprint arXiv:2002.02574},
  year={2020}
}

@inproceedings{goel2019beyond,
  title={Beyond online balanced descent: An optimal algorithm for smoothed online optimization},
  author={Goel, Gautam and Lin, Yiheng and Sun, Haoyuan and Wierman, Adam},
  booktitle={Advances in Neural Information Processing Systems},
  pages={1875--1885},
  year={2019}
}

@article{goel2021competitive,
  title={Competitive Control},
  author={Goel, Gautam and Hassibi, Babak},
  journal={arXiv preprint arXiv:2107.13657},
  year={2021}
}

@inproceedings{martin2022safe,
  title={Safe control with minimal regret},
  author={Martin, Andrea and Furieri, Luca and D{\"o}rfler, Florian and Lygeros, John and Ferrari-Trecate, Giancarlo},
  booktitle={Learning for Dynamics and Control Conference},
  pages={726--738},
  year={2022},
  organization={PMLR}
}

@inproceedings{goel2022online,
  title={Online Estimation and Control with Optimal Pathlength Regret},
  author={Goel, Gautam and Hassibi, Babak},
  booktitle={Learning for Dynamics and Control Conference},
  pages={404--414},
  year={2022},
  organization={PMLR}
}

@phdthesis{goel2022regret,
  title={Regret-Optimal Control},
  author={Goel, Gautam},
  year={2022},
  school={California Institute of Technology}
}

@book{kailath2000linear,
  title={Linear estimation},
  author={Kailath, Thomas and Sayed, Ali H and Hassibi, Babak},
  year={2000},
  publisher={Prentice Hall}
}

@book{hassibi1999indefinite,
  title={Indefinite-quadratic estimation and control: a unified approach to H 2 and H-infinity theories},
  author={Hassibi, Babak and Sayed, Ali H and Kailath, Thomas},
  year={1999},
  publisher={SIAM}
}

\section{Appendix}

\subsection{Missing Lemmas}
We present proofs of some of the key lemmas used in the main body.

\begin{lemma} \label{canonical-factorization-lemma}
The following canonical factorization holds: $$ I + F(z) F(z^{-*})^* = \Delta(z) \Delta^*(z^{-*}), $$ where we define 
\begin{equation} \label{delta-definition}
\Delta(z) = (I + L(zI - A)^{-1}K )\Sigma^{1/2},
\end{equation}
$$
K = APL\Sigma^{-1}, \hspace{5mm} \Sigma = I + LPL,$$
and $P$ is the unique Hermitian solution to the Riccati equation
$$P = B_u B_u^* + APA^* - APL (I + LPL)^{-1} LPA^*.$$
\end{lemma}

\begin{proof}
We expand $ I + F(z) F(z^{-*})^*$ as
$$  \begin{bmatrix} L(zI - A)^{-1} & I \end{bmatrix} \begin{bmatrix} B_uB_u^* & 0 \\ 0 & I \end{bmatrix} \begin{bmatrix} (z^{-*}I - A)^{-*}L \\ I \end{bmatrix}.$$
Applying Lemma \ref{equivalence-class-lemma-2}, we see that this equals
$$ \begin{bmatrix} L(zI - A)^{-1} & I \end{bmatrix} \Lambda(P) \begin{bmatrix} (z^{-*}I - A)^{-*}L \\ I \end{bmatrix},$$
where $P$ is an arbitrary Hermitian matrix and we define $$\Lambda(P) = \begin{bmatrix} B_uB_u^* - P + APA^* & APL \\ LPA^* & I + LPL \end{bmatrix}.$$ Notice that the $\Lambda(P)$ can be factored as $$\begin{bmatrix} I & K(P) \\ 0 & I \end{bmatrix} \begin{bmatrix} \Gamma(P) & 0 \\ 0 & \Sigma(P) \end{bmatrix} \begin{bmatrix} I & 0 \\ K^*(P) & I \end{bmatrix}, $$
where we define $$\Gamma(P) = B_u B_u^* - P + APA^* - K(P)\Sigma(P) K^*(P),$$ 
$$K(P) = APL\Sigma(P)^{-1},$$
$$\Sigma(P) = I + LPL.$$
By assumption, $(A, B_u)$ is stabilizable and $(A, L)$ is detectable, therefore the Riccati equation $\Gamma(P) = 0$ has a unique stabilizing solution (see, e.g. Theorem E.6.2 in \cite{kailath2000linear}).  Suppose $P$ is chosen to be this solution, and define $K = K(P)$, $\Sigma = \Sigma(P)$. We immediately obtain the canonical factorization $$ I + F(z) F(z^{-*})^* = \Delta(z) \Delta^*(z^{-*}), $$ where we define 
\begin{equation} \label{delta-definition}
\Delta(z) = (I + L(zI - A)^{-1}K )\Sigma^{1/2}.
\end{equation}
\end{proof}

\begin{lemma} \label{equivalence-class-lemma-1}
For all $H, F$ and all Hermitian matrices $P$, we have
$$\begin{bmatrix} H^*(z^{-1}I - F^*)^{-1} & I \end{bmatrix} \Omega(P) \begin{bmatrix} (zI - F)^{-1}H \\ I \end{bmatrix} = 0, $$ where we define $$\Omega(P) = \begin{bmatrix} -P + F^*PF & F^*PH \\ H^*PF & H^*PH \end{bmatrix}. $$
\end{lemma}

\begin{proof}
This identity is essentially the ``transpose'' of Lemma \ref{equivalence-class-lemma-2} and is easily verified via direct calculation.
\end{proof}

\begin{lemma}\label{equivalence-class-lemma-2}
For all $H, F$ and all Hermitian matrices $P$, we have
$$\begin{bmatrix} H(zI - F)^{-1} & I \end{bmatrix} \Omega(P) \begin{bmatrix} (z^{-1}I - F^*)^{-1}H^* \\ I \end{bmatrix} = 0, $$ where we define $$\Omega(P) = \begin{bmatrix} -P + FPF^* & FPH^* \\ HPF^* & HPH^* \end{bmatrix}. $$
\end{lemma}

\begin{proof}
This identity is a special case of Lemma \ref{general-equivalence-class-lemma}; it also appears as Lemma 12.3.3 in ``Indefinite-Quadratic Estimation and Control'' by Hassibi, Sayed, and Kailath.
\end{proof}

\begin{lemma} \label{general-equivalence-class-lemma}
For all $H_1, H_2, F_1, F_2$ and all matrices $W$, we have
$$\begin{bmatrix} H_1(zI - F_1)^{-1} & I \end{bmatrix} \Omega(W) \begin{bmatrix} (z^{-1}I - F_2^*)^{-1}H_2^* \\ I \end{bmatrix} = 0, $$ where we define $$\Omega(W) = \begin{bmatrix} -W + F_1WF_2^* & F_1WH_2^* \\ H_1WF_2^* & H_1WH_2^* \end{bmatrix}. $$
\end{lemma}

\begin{proof}
Notice that $\Omega(W)$ can be rewritten as $$\Omega(W) = \begin{bmatrix} F_1 \\ H_1 \end{bmatrix} W \begin{bmatrix} F_2 & H_2 \end{bmatrix} - \begin{bmatrix} I \\0 \end{bmatrix} W \begin{bmatrix} I & 0 \end{bmatrix}.$$
The proof is immediate after observing that $$\begin{bmatrix} H_1(zI - F_1)^{-1} & I \end{bmatrix}\begin{bmatrix} F_1 \\ H_1 \end{bmatrix} = H_1(zI - F_1)^{-1}z,$$ 
$$\begin{bmatrix} F_2 & H_2 \end{bmatrix} \begin{bmatrix} (z^{-1}I - F_2^*)^{-1}H_2^* \\ I \end{bmatrix} = z^{-1}(z^{-1}I - F_2^*)^{-1}H_2^*.$$
\end{proof}

\subsection{A state-space model for the $H_{\infty}$-optimal measurement-feedback controller}

\begin{theorem}[Theorem 13.3.5 in \cite{hassibi1999indefinite}] \label{hinf-measurement-feedback-thm}
A causal measurement-feedback controller $K$ such that $$\|T_K\| < 1 $$ exists if and only if the control DARE $$P_c = A^*P_cA + L^*L - K_c^* R_c K_c$$ and the estimation DARE 
$$P_e = AP_e A^* +    B_w B_w^* - K_e R_e K_e^*,$$ where we define 
$$\begin{cases}
K_c = R_c^{-1}\begin{bmatrix} B_u^* \\  B_w^* \end{bmatrix}P_c A \\
\\
R_c = \begin{bmatrix} I_m & 0 \\ 0 & -I_p \end{bmatrix} + \begin{bmatrix} B_u^* \\ 
B_w^* \end{bmatrix}P_c \begin{bmatrix} B_u &    B_w \end{bmatrix}, \\
\\
K_e = AP^d \begin{bmatrix}  C^* & L^* \end{bmatrix}R_e^{-1}, \\
\\
R_e = \begin{bmatrix} I_r & 0 \\ 0 & -I_n \end{bmatrix} + \begin{bmatrix}  C \\ L \end{bmatrix}P_c \begin{bmatrix}  C^* & L^* \end{bmatrix}, 
\end{cases} $$  have solutions $P_c \succeq 0$ and $P_e \succeq 0$ such that 
\begin{enumerate}
    \item The matrix $ A - \begin{bmatrix} B_u &    B_w \end{bmatrix}K_c$ is stable.
    \item The matrix $R_c$ has $m$ positive eigenvalues and $p$ negative eigenvalues.
    \item The matrix $A - K_e\begin{bmatrix}  C \\ L \end{bmatrix}$ is stable.
    \item The matrix $R_e$ has $r$ positive eigenvalues and $n$ negative eigenvalues.
    \item $\rho(P_c P_e) < 1$.
\end{enumerate}
If these conditions are satisfied, then one possible choice of $K$ is given by $$u_t = - K_u (\widehat{x}_t + P C^*(I_r +  C P C^*)^{-1}(y_t - C \widehat{x}_t)),$$ where the state-estimate $\widehat{x}_t$ is given by the recursion $$\widehat{x}_{t+1} = (A - B_wK_w)(\widehat{x}_t + P C^*(I_r +  C P C^*)^{-1}(y_t - C \widehat{x}_t)) + B_u u_t,$$ and $K_u \in \mathbb{R}^{m \times n}$ and $K_w \in \mathbb{R}^{p \times n}$ are defined as $$\begin{bmatrix} K_u \\ K_w \end{bmatrix} = K_c,$$ and we define $P \in \mathbb{R}^{n \times n}$ as $$ P = P_e(I_n - P_c P_e)^{-1}. $$
\end{theorem}

While Theorem \ref{hinf-measurement-feedback-thm} tells us how to determine whether there exists a controller $K$ such that $\|T_K\| < 1$, it does not directly answer the more general question if there exists a controller $K$ such that $\|T_K\| < \gamma$, for any fixed $\gamma > 0$. This condition is equivalent to $\|T_{\widehat{K}}\| < 1$, where we define $T_{\widehat{K}} = \gamma^{-1}T_K$. Notice that
$$ T_{\widehat{K}} =  \begin{bmatrix} \gamma^{-1}G & 0 \\ 0 & 0 \end{bmatrix}  + \begin{bmatrix}F \\ I \end{bmatrix} (\gamma^{-1}Q) \begin{bmatrix} J & I \end{bmatrix}.$$
Recall that $Q = K(I - HK)^{-1}$, implying that $\gamma^{-1}Q = (\gamma^{-1}K) [I - (\gamma H) (\gamma^{-1}K)]^{-1}$. It follows that the controller $K$ satisfying $\|T_K\| < \gamma$ in the system $F, G, H, J$ is precisely $\gamma \widehat{K}$, where $\widehat{K}$ is the controller satisfying $\|T_{\widehat{K}}\| < 1$ in the system $\widehat{F}, \widehat{G}, \widehat{J}, \widehat{H}$ and $\widehat{F} = F, \widehat{G} = \gamma^{-1}G, \widehat{H} = \gamma H, \widehat{J} = J$.  We can assign state-space structure to $\widehat{F}, \widehat{G}, \widehat{H}, \widehat{J}$ as follows. Recall that $$F(z) = L(zI - A)^{-1}B_u, \hspace{3mm} G(z) = L(zI - A)^{-1}B_w,$$ $$H(z) = C(zI - A)^{-1}B_u, \hspace{3mm} J(z) = C(zI - A)^{-1}B_w.$$ Define $$\widehat{C} =  \gamma C,  \hspace{3mm}  \widehat{B}_w = \gamma^{-1}B_w.$$ We have
$$\widehat{F}(z) =  L(zI - A)^{-1}\widehat{B}_u, \hspace{3mm} \widehat{G}(z) = L(zI - A)\widehat{B}_w, $$ $$ \widehat{H}(z) = \widehat{C}(zI - A)^{-1}B_u,  \hspace{3mm} \widehat{J}(z) = \widehat{C}(zI - A)\widehat{B}_w.$$  We can hence use Theorem \ref{hinf-measurement-feedback-thm} to check the existence of a controller $K$ such that $\|T_K\| < \gamma$, for any $\gamma > 0$.

\end{document}